\documentclass[reqno]{amsart}

\usepackage{amsmath, amssymb}
\usepackage[mathscr]{eucal}
\usepackage{hyperref}
\hypersetup{colorlinks=true, linkcolor=blue, urlcolor=blue, citecolor=[rgb]{0, 0.8, 0}}

\theoremstyle{plain}
\newtheorem{theorem}{Theorem}[section]
\newtheorem{lemma}[theorem]{Lemma}
\newtheorem{corollary}[theorem]{Corollary}

\numberwithin{equation}{section}

\newcommand{\du}{\mathrm{d}}

\newcommand{\iu}{\mathrm{i}}

\DeclareMathOperator{\im}{Im}
\DeclareMathOperator{\ind}{ind}
\DeclareMathOperator{\Res}{Res}

\title{On two-spectra inverse problems}

\author{Namig J. Guliyev}
\address{Institute of Mathematics and Mechanics, Azerbaijan National Academy of Sciences, 9 B.~Vahabzadeh str., AZ1141, Baku, Azerbaijan.}
\email{njguliyev@gmail.com}

\subjclass[2010]{34A55, 34B07, 34B24, 34L40, 47A75, 47E05}

\keywords{two-spectra inverse problem, one-dimensional Schr\"{o}dinger equation, boundary conditions dependent on the eigenvalue parameter}

\begin{document}
\maketitle
\begin{abstract}
We consider a two-spectra inverse problem for the one-dimensional Schr\"{o}dinger equation with boundary conditions containing rational Herglotz--Nevanlinna functions of the eigenvalue parameter and provide a complete solution of this problem.
\end{abstract}

\tableofcontents

\section{Introduction and main result} \label{sec:introduction}

The study of two-spectra inverse problems was initiated by Borg~\cite{B46}, who proved that the potential $q$ of the one-dimensional Schr\"{o}dinger equation
\begin{equation} \label{eq:SL}
  -y''(x) + q(x)y(x) = \lambda y(x)
\end{equation}
is uniquely determined by the spectra of the boundary value problems generated by this equation and the boundary conditions $y'(0) = h_1 y(0)$, $y'(\pi) = H y(\pi)$ and $y'(0) = h_2 y(0)$, $y'(\pi) = H y(\pi)$ respectively (with $h_1 \ne h_2$). Subsequent developments by Marchenko~\cite{M52}, Krein~\cite{K51}, Levitan and Gasymov~\cite{L64}, ~\cite{LG64} and others showed that not only the potential $q$ but also the boundary coefficients $h_1$, $h_2$ and $H$ are uniquely determined by these spectra, and that any two interlacing sequences satisfying certain asymptotic conditions are indeed the spectra of boundary value problems of the above form (see also \cite{M77}). These results were relatively recently generalized to problems with distributional potentials \cite{EGNT13}, \cite{HM04}, \cite{SS05}.

In this paper we are interested in two-spectra inverse problems for boundary value problems with boundary conditions dependent on the eigenvalue parameter. Such problems have also been considered in the literature. Some uniqueness results were obtained in~\cite{AOK09}, \cite{BP80}, \cite{BBW04}. The papers~\cite{C01}, \cite{M82} contain some existence results for problems with one eigenvalue-parameter-dependent boundary condition. In the case when only one of the boundary conditions depends linearly on the eigenvalue parameter, necessary and sufficient conditions for solvability of the two-spectra inverse problem were found in~\cite{K04}, \cite{G05}. For problems with coupled boundary conditions dependent on the eigenvalue parameter, see~\cite{IN17} and the references therein.

We consider two-spectra inverse problems for boundary value problems generated by the equation~(\ref{eq:SL}) together with boundary conditions of the form
\begin{equation} \label{eq:boundary}
  \frac{y'(0)}{y(0)} = -f(\lambda), \qquad \frac{y'(\pi)}{y(\pi)} = F(\lambda),
\end{equation}
where $q \in \mathscr{L}_2(0, \pi)$ is real-valued, $f$ is a rational Herglotz--Nevanlinna function, and $F$ is either a rational Herglotz--Nevanlinna function or the symbol $\infty$ to denote the Dirichlet condition at the right endpoint. A boundary condition from~(\ref{eq:boundary}) is also interpreted as the Dirichlet condition whenever $\lambda$ is a pole of the corresponding function on the right-hand side. Using Darboux-type transformations between such boundary value problems, we recently obtained in~\cite{G17} various direct and inverse spectral results for boundary value problems of the form~(\ref{eq:SL})-(\ref{eq:boundary}), and then extended these results to operators with distributional potentials~\cite{G19} and Bessel-type singularities~\cite{G20b}. But these transformations are not applicable to two-spectra inverse problems because a pair of boundary value problems with a common boundary condition is transformed to a pair of boundary value problems with no common boundary conditions. Therefore we first reduce our two-spectra problem to an inverse problem solved in~\cite{G17} and then completely solve the two-spectra problem.

We denote the boundary value problem (\ref{eq:SL})-(\ref{eq:boundary}) by $\mathscr{P}(q, f, F)$ and use the notation
\begin{equation*}
  x_n = y_n + \ell_2 \left( \frac{1}{n} \right)
\end{equation*}
to mean $\sum_{n = 0}^\infty \left| n (x_n - y_n) \right|^2 < \infty$. It is easy to see that an eigenvalue of $\mathscr{P}(q, f, F)$ is also an eigenvalue of $\mathscr{P}(q, f + \alpha, F)$ with $\alpha \ne 0$ if and only if this eigenvalue is a pole of $f$. Throughout this paper we assume that no eigenvalue of $\mathscr{P}(q, f, F)$ is a pole of $f$ and thus the spectra of the problems $\mathscr{P}(q, f, F)$ and $\mathscr{P}(q, f + \alpha, F)$ are disjoint. It turns out that the problems $\mathscr{P}(q, f, F)$ and $\mathscr{P}(q, f + \alpha, F)$ are not, in general, uniquely determined by their spectra. However, we are able to describe all pairs of problems with the given two spectra. Our main result reads as follows.

\begin{theorem} \label{thm:two_spectra}
  Two disjoint sequences $\{ \lambda_n \}_{n \ge 0}$ and $\{ \mu_n \}_{n \ge 0}$ are the eigenvalues of a pair of problems of the form $\mathscr{P}(q, f, F)$ and $\mathscr{P}(q, f + \alpha, F)$ if and only if they interlace and satisfy asymptotics of the form
\begin{equation*}
  \sqrt{\lambda_n} = n - L + \frac{\sigma}{\pi n} + \ell_2 \left( \frac{1}{n} \right), \qquad \sqrt{\lambda_n} - \sqrt{\vphantom{\lambda_n} \mu_n} = (n - L)^{-2 r - 1} \left( \nu + \ell_2 \left( \frac{1}{n} \right) \right)
\end{equation*}
for some integer or half-integer $L \ge -1/2$, $\sigma \in \mathbb{R}$, $\nu \in \mathbb{R} \setminus \{0\}$ and $r \in \{0, 1\}$, with the exception of the case when $L = -1/2$ and $r = 1$. Moreover, there is a one-to-one correspondence between such problems and sets of nonnegative integers of cardinality not exceeding $L + (1 - r) / 2$.
\end{theorem}

In particular, the proof of this theorem (Section~\ref{sec:inverse}) also yields the following uniqueness result.

\begin{corollary} \label{cor:uniqueness}
  The problems $\mathscr{P}(q, f, F)$ and $\mathscr{P}(q, f + \alpha, F)$ are uniquely determined by their spectra and the poles of $f$.
\end{corollary}

So in a sense, the amount of spectral data required for the unique determination in our case (two spectra and a finite number of indices) is between that of the classical case (two spectra only) and of problems with coupled boundary conditions (two spectra and an infinite sequence of signs); for the latter case see, e.g.,~\cite{GN95} and the references therein.

The paper is organized as follows. In Section~\ref{sec:preliminaries} we introduce the necessary notation and obtain some useful identities. In Section~\ref{sec:two_problems} we find some conditions for the eigenvalues of the problems $\mathscr{P}(q, f, F)$ and $\mathscr{P}(q, f + \alpha, F)$. Section~\ref{sec:inverse} is devoted to the proof of our main result, Theorem~\ref{thm:two_spectra}. Finally, in the appendix we prove two auxiliary lemmas used in the main text.

\section{Preliminaries} \label{sec:preliminaries}

First we introduce some further notation. Since $f$ is a rational Herglotz--Nevan\-linna function, it can be written as
\begin{equation*}
  f(\lambda) = h_0 \lambda + h + \sum_{k=1}^d \frac{\delta_k}{h_k - \lambda}
\end{equation*}
with $h_0 \ge 0$, $h \in \mathbb{R}$, $\delta_k > 0$, and $h_1 < \ldots < h_d$. We assign to $f$ two polynomials $f_\uparrow$ and $f_\downarrow$ by writing this function as
$$
  f(\lambda) = \frac{f_\uparrow(\lambda)}{f_\downarrow(\lambda)},
$$
where
$$
  f_\downarrow(\lambda) := h'_0 \prod_{k=1}^d (h_k - \lambda), \qquad h'_0 := \begin{cases} 1 / h_0, & h_0 > 0, \\ 1, & h_0 = 0. \end{cases}
$$
We define the \emph{index} of $f$ as
$$
  \ind f := \begin{cases} 2 d + 1, & h_0 > 0, \\ 2 d, & h_0 = 0. \end{cases}
$$
The number $\ind F$ and the polynomials $F_\downarrow$ and $F_\uparrow$ are defined in a similar way when $F$ is also a rational Herglotz--Nevanlinna function. If $F = \infty$ then we just set
$$
  F_\uparrow(\lambda) := -1, \qquad F_\downarrow(\lambda) := 0, \qquad \ind F := -1.
$$

Let $\varphi(x, \lambda)$, $\psi(x, \lambda)$ and $\chi(x, \lambda)$ be the solutions of~(\ref{eq:SL}) satisfying the initial conditions
\begin{equation} \label{eq:phi_psi}
\begin{aligned}
  \varphi(0, \lambda) &= f_\downarrow(\lambda), & \psi(0, \lambda) &= f_\downarrow(\lambda), & \chi(\pi, \lambda) &= F_\downarrow(\lambda), \\
  \varphi'(0, \lambda) &= -f_\uparrow(\lambda), \quad & \psi'(0, \lambda) &= -f_\uparrow(\lambda) - \alpha f_\downarrow(\lambda), \quad & \chi'(\pi, \lambda) &= F_\uparrow(\lambda).
\end{aligned}
\end{equation}
Then the eigenvalues of the boundary value problems $\mathscr{P}(q, f, F)$ and $\mathscr{P}(q, f + \alpha, F)$ coincide with the zeros of (their \emph{characteristic functions})
\begin{equation*}
  \Phi(\lambda) := F_\uparrow(\lambda) \varphi(\pi, \lambda) - F_\downarrow(\lambda) \varphi'(\pi, \lambda) = f_\downarrow(\lambda) \chi'(0, \lambda) + f_\uparrow(\lambda) \chi(0, \lambda)
\end{equation*}
and
\begin{equation*}
  \Psi(\lambda) := F_\uparrow(\lambda) \psi(\pi, \lambda) - F_\downarrow(\lambda) \psi'(\pi, \lambda) = f_\downarrow(\lambda) \chi'(0, \lambda) + \left( f_\uparrow(\lambda) + \alpha f_\downarrow(\lambda) \right) \chi(0, \lambda)
\end{equation*}
respectively. These eigenvalues have the asymptotics (see \cite[Theorem 4.2]{G17} for details)
\begin{equation} \label{eq:lambda}
  \sqrt{\lambda_n} = n - \frac{\ind f + \ind F}{2} + \frac{\sigma}{\pi n} + \ell_2 \left( \frac{1}{n} \right)
\end{equation}
and
\begin{equation*}
  \sqrt{\mu_n} = n - \frac{\ind f + \ind F}{2} + \frac{\sigma'}{\pi n} + \ell_2 \left( \frac{1}{n} \right)
\end{equation*}
with
\begin{equation*}
  \sigma - \sigma' = \begin{cases} \alpha, & \ind f \text{ is even} \\ 0, & \ind f \text{ is odd.} \end{cases}
\end{equation*}
In the next section we will obtain more refined asymptotics for the difference of the square roots of the eigenvalues $\lambda_n$ and $\mu_n$.

Since for each eigenvalue $\lambda_n$ of $\mathscr{P}(q, f, F)$ the solutions $\varphi(x, \lambda_n)$ and $\chi(x, \lambda_n)$ are linearly dependent, there exists a unique number $\beta_n \ne 0$ such that
\begin{equation} \label{eq:beta}
  \chi(x, \lambda_n) = \beta_n \varphi(x, \lambda_n).
\end{equation}
The boundary value problem~(\ref{eq:SL})-(\ref{eq:boundary}) is equivalent to an eigenvalue problem for a self-adjoint operator in a Hilbert space of the form $\mathscr{L}_2(0, \pi) \oplus \mathbb{C}^k$, in the sense that their eigenvalues coincide and the eigenfunctions of~(\ref{eq:SL})-(\ref{eq:boundary}) are in one-to-one correspondence with the first components of the eigenvectors of the self-adjoint operator~\cite[Section 2.2]{G17} (see also~\cite{G20a} for a generalization of this technique to the case of not necessarily rational Herglotz--Nevanlinna functions). For each nonnegative integer $n$, the \emph{norming constant} $\gamma_n$ is defined as the squared norm of the eigenvector of this operator whose first component coincides with $\varphi(x, \lambda_n)$:
\begin{equation*}
  \gamma_n := \int_0^{\pi} \varphi^2(x, \lambda_n) \,\du x + f'(\lambda_n) f_\downarrow^2(\lambda_n) + \frac{F'_\uparrow(\lambda_n) F_\downarrow(\lambda_n) - F_\uparrow(\lambda_n) F'_\downarrow(\lambda_n)}{\beta_n^2},
\end{equation*}
where $f'(\lambda_n)$ is well defined by our assumption that no eigenvalue is a pole of $f$. These norming constants have the asymptotics (\cite[Theorem 4.2]{G17})
\begin{equation} \label{eq:gamma}
  \gamma_n = \frac{\pi}{2} \left( n - \frac{\ind f + \ind F}{2} \right)^{2 \ind f} \left( 1 + \ell_2 \left( \frac{1}{n} \right) \right).
\end{equation}
The sequences $\{ \lambda_n \}_{n \ge 0}$, $\{ \beta_n \}_{n \ge 0}$ and $\{ \gamma_n \}_{n \ge 0}$ are related by the identity (\cite[Lemma~2.1]{G17})
\begin{equation} \label{eq:Phi_beta_gamma}
  \Phi'(\lambda_n) = \beta_n \gamma_n.
\end{equation}

In the remaining part of this section, we are going to show that the coefficients of the polynomial $f_\downarrow(\lambda)$ satisfy a nonsingular system of linear equations whose coefficients are expressed in terms of the sequences $\{ \lambda_n \}_{n \ge 0}$ and $\{ \gamma_n \}_{n \ge 0}$. Thus any polynomial whose coefficients satisfy this system must necessarily coincide with $f_\downarrow(\lambda)$. We will need this result in Section~\ref{sec:inverse}.

We start with some identities for the eigenvalues and the norming constants of the problem $\mathscr{P}(q, f, F)$. Such identities are characteristic to problems with boundary conditions dependent on the eigenvalue parameter; they were used in~\cite{G06} to obtain explicit expressions for all the coefficients of the boundary conditions in the case of linear dependence on the eigenvalue parameter (i.e., $\ind f = \ind F = 2$ in our notation).

\begin{lemma} \label{lem:zero_sum}
The following identities hold:
\begin{equation*}
  \sum_{n=0}^\infty \frac{\lambda_n^k f_\downarrow(\lambda_n)}{\gamma_n} = 0, \qquad k = 0, \ldots, d - 1.
\end{equation*}
\end{lemma}
\begin{proof}
From~(\ref{eq:phi_psi}) and~(\ref{eq:beta}) we have
\begin{equation*}
  f_\downarrow(\lambda_n) = \varphi(0, \lambda_n) = \frac{\chi(0, \lambda_n)}{\beta_n}.
\end{equation*}
Together with~(\ref{eq:Phi_beta_gamma}) this implies (for sufficiently large $N$)
\begin{equation*}
  \sum_{n=0}^N \frac{\lambda_n^k f_\downarrow(\lambda_n)}{\gamma_n} = \sum_{n=0}^N \Res_{\lambda = \lambda_n} \frac{\lambda^k \chi(0, \lambda)}{\Phi(\lambda)} = \frac{1}{2 \pi \iu} \int_{C_N} \frac{\lambda^k \chi(0, \lambda)}{\Phi(\lambda)} \,\du \lambda,
\end{equation*}
where $C_N$ is the circle of radius
\begin{equation*}
  \left( N - \frac{\ind f + \ind F - 1}{2} \right)^2
\end{equation*}
centered at the origin. Expressing $\chi(x, \lambda)$ as a linear combination of the cosine- and sine-type solutions we obtain $\chi(x, \lambda) = O \left( \left| \sqrt{\lambda} \right|^{\ind F} e^{|\im \sqrt{\lambda}\pi|} \right)$. On the other hand, from~(\ref{eq:Phi_rough}) we get (see, e.g., the proof of \cite[Lemma 5.2]{G05} for details)
\begin{equation*}
  \frac{1}{\Phi(\lambda)} = O \left( \left| \sqrt{\lambda} \right|^{-(\ind f + \ind F + 1)} e^{-|\im \sqrt{\lambda}\pi|} \right), \qquad \lambda \in \bigcup_N C_N,
\end{equation*}
and thus
\begin{equation*}
  \frac{\lambda^k \chi(0, \lambda)}{\Phi(\lambda)} = O \left( \frac{1}{N^{\ind f - 2 k + 1}} \right), \qquad \lambda \in \bigcup_N C_N
\end{equation*}
with $\ind f - 2 k + 1 \ge 3$. Hence
\begin{equation*}
  \lim_{N \to \infty} \int_{C_N} \frac{\lambda^k \chi(0, \lambda)}{\Phi(\lambda)} \,\du \lambda = 0,
\end{equation*}
which proves the lemma.
\end{proof}

Denote by $p_{d-1}$, $\ldots$, $p_0$ the non-leading coefficients of the polynomial $f_\downarrow(\lambda)$ after dividing it by its leading coefficient:
\begin{equation*}
  \frac{(-1)^d}{h'_0} f_\downarrow(\lambda) = \prod_{k=1}^d (\lambda - h_k) = \lambda^d + p_{d-1} \lambda^{d-1} + \ldots + p_1 \lambda + p_0.
\end{equation*}
It is easy to see from the asymptotics of the eigenvalues and the norming constants that for each $k = 0$, $\ldots$, $2 d - 1$ the series
\begin{equation*}
  s_k := \sum_{n=0}^\infty \frac{\lambda_n^k}{\gamma_n}
\end{equation*}
converges absolutely. Lemma~\ref{lem:zero_sum} implies the following identities between the numbers $p_i$ and $s_j$:
\begin{equation} \label{eq:p_s}
  \sum_{i=0}^{d-1} p_i s_{i+k} = -s_{d+k}, \qquad k = 0, 1, \ldots, d - 1.
\end{equation}
We consider them as a system of linear equations (with respect to the numbers $p_i$), the matrix of which is the following Hankel matrix:
\begin{equation*}
  \begin{pmatrix}
    s_0     & s_1    & \ldots & s_{d-1} \\
    s_1     & s_2    & \ldots & s_d \\
    \vdots  & \vdots & \ddots & \vdots \\
    s_{d-1} & s_d    & \ldots & s_{2d-2}
  \end{pmatrix}.
\end{equation*}
The quadratic form corresponding to this matrix is positive definite:
\begin{equation*}
  \sum_{i,j=0}^{d-1} s_{i+j} \xi_i \xi_j = \sum_{i,j=0}^{d-1} \sum_{n=0}^\infty \frac{\lambda_n^{i+j} \xi_i \xi_j}{\gamma_n} = \sum_{n=0}^\infty \sum_{i,j=0}^{d-1} \frac{\lambda_n^{i+j} \xi_i \xi_j}{\gamma_n} = \sum_{n=0}^\infty \frac{1}{\gamma_n} \left( \sum_{i=0}^{d-1} \lambda_n^i \xi_i \right)^2 \ge 0
\end{equation*}
with equality if and only if $\sum_{i=0}^{d-1} \lambda_n^i \xi_i = 0$ for all $n$, i.e. $\xi_0 = \ldots = \xi_{d-1} = 0$. Thus the determinant of the above matrix is strictly positive and hence the system~(\ref{eq:p_s}) has a unique solution.

\section{Properties of two problems with a common boundary condition} \label{sec:two_problems}

We are now going to study further properties of the sequences $\{ \lambda_n \}_{n \ge 0}$ and $\{ \mu_n \}_{n \ge 0}$. We will first show that these two sequences interlace and then find more refined asymptotics for the difference of their square roots. As we will see in the next section, any two sequences with these two properties are indeed the eigenvalues of a pair of boundary value problems with a common boundary condition.

The function
\begin{equation*}
  m(\lambda) := -\frac{\Psi(\lambda)}{\Phi(\lambda)}
\end{equation*}
satisfies the identity $m(\overline{\lambda}) = \overline{m(\lambda)}$ and is a meromorphic function with poles at $\lambda_n$ and zeros at $\mu_n$. For nonreal values of $\lambda$ the solution
\begin{equation*}
  y(x, \lambda) := \psi(x, \lambda) + m(\lambda) \varphi(x, \lambda)
\end{equation*}
satisfies the boundary condition
\begin{equation*}
  F_\uparrow(\lambda) y(\pi, \lambda) - F_\downarrow(\lambda) y'(\pi, \lambda) = 0.
\end{equation*}
Using~(\ref{eq:phi_psi}) we calculate
\begin{multline*}
  \left. (\lambda - \mu) \int_0^{\pi} y(x, \lambda) y(x, \mu) \,\du x = \left( y(x, \lambda) y'(x, \mu) - y'(x, \lambda) y(x, \mu) \right) \right|_0^{\pi} \\
  = \left( F(\mu) - F(\lambda) \right) y(\pi, \lambda) y(\pi, \mu) + \alpha f_\downarrow(\lambda) f_\downarrow(\mu) \left( m(\lambda) - m(\mu) \right) \\
  + \left( f_\downarrow(\lambda) f_\uparrow(\mu) - f_\downarrow(\mu) f_\uparrow(\lambda) \right) \left( 1 + m(\lambda) \right) \left( 1 + m(\mu) \right).
\end{multline*}
For $\mu = \overline{\lambda}$ this yields
\begin{equation*}
\begin{split}
  \alpha \frac{\im m(\lambda)}{\im \lambda} &= \frac{1}{\left| f_\downarrow(\lambda) \right|^2} \int_0^{\pi} \left| y(x, \lambda) \right|^2 \,\du x \\
  &+ \left| \frac{y(\pi, \lambda)}{f_\downarrow(\lambda)} \right|^2 \frac{\im F(\lambda)}{\im \lambda} + \left| 1 + m(\lambda) \right|^2 \frac{\im f(\lambda)}{\im \lambda} > 0.
\end{split}
\end{equation*}
Thus $\alpha m(\lambda)$ is a Herglotz--Nevanlinna function, and hence its zeros $\mu_n$ and poles $\lambda_n$ interlace.

Using~(\ref{eq:phi_psi}), (\ref{eq:beta}) and the constancy of the Wronskian we obtain
\begin{align*}
  \Psi(\lambda_n) &= F_\uparrow(\lambda_n) \psi(\pi, \lambda_n) - F_\downarrow(\lambda_n) \psi'(\pi, \lambda_n) \\
  &= \beta_n \left( \varphi'(\pi, \lambda_n) \psi(\pi, \lambda_n) - \varphi(\pi, \lambda_n) \psi'(\pi, \lambda_n) \right) \\
  &= \beta_n \left( \varphi'(0, \lambda_n) \psi(0, \lambda_n) - \varphi(0, \lambda_n) \psi'(0, \lambda_n) \right) = \alpha \beta_n f_\downarrow^2(\lambda_n).
\end{align*}
Together with~(\ref{eq:Phi_beta_gamma}) this implies
\begin{equation*}
  \gamma_n = \frac{\alpha f_\downarrow^2(\lambda_n) \Phi'(\lambda_n)}{\Psi(\lambda_n)}.
\end{equation*}
We will need this formula in the next section in order to transform our two-spectra inverse problem to an inverse problem solved in~\cite{G17}, but for now we will use it to obtain more refined asymptotics for the difference $\sqrt{\lambda_n} - \sqrt{\vphantom{\lambda_n} \mu_n}$. The mean value theorem yields
\begin{equation} \label{eq:Psi_lambda_mu}
  \Psi(\lambda_n) = \Psi(\lambda_n) - \Psi(\mu_n) = \left( \sqrt{\lambda_n} - \sqrt{\vphantom{\lambda_n} \mu_n} \right) \left( \sqrt{\lambda_n} + \sqrt{\vphantom{\lambda_n} \mu_n} \right) \Psi'(\zeta_n)
\end{equation}
for some $\zeta_n \in \left[ \lambda_n, \mu_n \right]$ with $\sqrt{\zeta_n} = n - (\ind f + \ind F) / 2 + O \left( \frac{1}{n} \right)$. Thus
\begin{equation*}
  \sqrt{\lambda_n} - \sqrt{\vphantom{\lambda_n} \mu_n} = \frac{\alpha f_\downarrow^2(\lambda_n) \Phi'(\lambda_n)}{\left( \sqrt{\lambda_n} + \sqrt{\vphantom{\lambda_n} \mu_n} \right) \gamma_n \Psi'(\zeta_n)}.
\end{equation*}
Applying Lemma~\ref{lem:infinite_product} to the problems $\mathscr{P}(q, f, F)$ and $\mathscr{P}(q, f + \alpha, F)$, and then applying Lemma~\ref{lem:G} to the functions $\Phi$ and $\Psi$ and using (\ref{eq:gamma}), we obtain the asymptotics
\begin{equation*}
  \sqrt{\lambda_n} - \sqrt{\vphantom{\lambda_n} \mu_n} = \left( n - \frac{\ind f + \ind F}{2} \right)^{-2 r - 1} \left( \frac{\alpha \left( h'_0 \right)^2}{\pi} + \ell_2 \left( \frac{1}{n} \right) \right),
\end{equation*}
where
\begin{equation*}
  r := \ind f - 2 d = \begin{cases} 1, & \ind f \text{ is odd,} \\ 0, & \ind f \text{ is even.} \end{cases}
\end{equation*}

\section{Inverse problem} \label{sec:inverse}

In this section, we will prove Theorem~\ref{thm:two_spectra}. The results of the previous section show that if two sequences $\{ \lambda_n \}_{n \ge 0}$ and $\{ \mu_n \}_{n \ge 0}$ are the eigenvalues of the problems $\mathscr{P}(q, f, F)$ and $\mathscr{P}(q, f + \alpha, F)$, then they interlace and satisfy asymptotics of the form
\begin{equation} \label{eq:lambda_mu}
  \sqrt{\lambda_n} = n - L + \frac{\sigma}{\pi n} + \ell_2 \left( \frac{1}{n} \right), \qquad \sqrt{\lambda_n} - \sqrt{\vphantom{\lambda_n} \mu_n} = (n - L)^{-2 r - 1} \left( \nu + \ell_2 \left( \frac{1}{n} \right) \right)
\end{equation}
where
\begin{equation*}
  L := \frac{\ind f + \ind F}{2} \ge d - \frac{1 - r}{2} \ge -\frac{1}{2}, \qquad \nu := \frac{\alpha \left( h'_0 \right)^2}{\pi} \ne 0.
\end{equation*}
Note also that if $L = -1/2$ then $\ind f = 0$, and consequently $r = 0$. We are now going to prove that these conditions are also sufficient for two sequences to be the eigenvalues of two such problems. However, unlike the case of constant boundary conditions, in order to determine these problems uniquely, we need some additional data. The identity $\Psi(\lambda) - \Phi(\lambda) = \alpha f_\downarrow(\lambda) \chi(0, \lambda)$ (see Section~\ref{sec:preliminaries}) shows that the zeros of $f_\downarrow$ are also zeros of $\Psi - \Phi$. As we will see shortly, they can be chosen arbitrarily among the zeros of $\Psi - \Phi$.

Let now $\{ \lambda_n \}_{n \ge 0}$ and $\{ \mu_n \}_{n \ge 0}$ be any two sequences such that they interlace and satisfy asymptotics of the form~(\ref{eq:lambda_mu}) for some integer or half-integer $L \ge -1/2$, real $\nu \ne 0$ and $\sigma$, and $r \in \{0, 1\}$. Define the functions
\begin{equation*}
  \Phi(\lambda) := -\prod_{n < L} (\lambda_n - \lambda) \prod_{n = L} \pi (\lambda_n - \lambda) \prod_{n > L} \frac{\lambda_n - \lambda}{(n - L)^2}
\end{equation*}
and
\begin{equation*}
  \Psi(\lambda) := -\prod_{n < L} (\mu_n - \lambda) \prod_{n = L} \pi (\mu_n - \lambda) \prod_{n > L} \frac{\mu_n - \lambda}{(n - L)^2},
\end{equation*}
where in each product $n$ runs over all nonnegative integers satisfying the given condition. Let $d$ be an integer with
\begin{equation*}
  0 \le d \le L + \frac{1 - r}{2},
\end{equation*}
and let $i_1$, $i_2$, $\ldots$, $i_d$ be integers (indices) with $0 \le i_1 < i_2 < \ldots < i_d$. Define the polynomial
\begin{equation*}
  p(\lambda) := \prod_{k=1}^d (\tau_{i_k} - \lambda),
\end{equation*}
where $\tau_0 < \tau_1 < \ldots$ are the zeros of the function $\Phi(\lambda) - \Psi(\lambda)$. Lemma~\ref{lem:G} implies
\begin{equation*}
  \Phi'(\lambda_n) = (-1)^n \left( n - L \right)^{2L} \left( \frac{\pi}{2} + \ell_2 \left( \frac{1}{n} \right) \right).
\end{equation*}
Using~(\ref{eq:Psi_lambda_mu}) and Lemma~\ref{lem:G} again we also have
\begin{equation*}
  \Psi(\lambda_n) = (-1)^{n} (n - L)^{2 L - 2 r} \left( \pi \nu + \ell_2 \left( \frac{1}{n} \right) \right).
\end{equation*}
From the asymptotics of $\sqrt{\lambda_n} - \sqrt{\vphantom{\lambda_n} \mu_n}$ and the fact that $\lambda_n$ and $\mu_n$ interlace it follows that if $\nu > 0$ (respectively, $\nu < 0$) then $\mu_n < \lambda_n < \mu_{n+1}$ (respectively, $\lambda_n < \mu_n < \lambda_{n+1}$) for each $n \ge 0$. Thus the numbers $\gamma_n$ defined by
\begin{equation*}
  \gamma_n := \frac{\pi \nu p^2(\lambda_n) \Phi'(\lambda_n)}{\Psi(\lambda_n)}
\end{equation*}
are all positive and have the asymptotics
\begin{equation*}
  \gamma_n = \left( n - L \right)^{4 d + 2 r} \left( \frac{\pi}{2} + \ell_2 \left( \frac{1}{n} \right) \right).
\end{equation*}
By \cite[Theorem 4.7]{G17}, there exists a boundary value problem $\mathscr{P}(q, f, F)$ having the eigenvalues $\{ \lambda_n \}_{n \ge 0}$ and the norming constants $\{ \gamma_n \}_{n \ge 0}$. Moreover, $\ind f = 2 d + r \ge 0$ and $\ind F = 2 L - 2 d - r \ge -1$. Denote $\alpha := \pi \nu / \left( h'_0 \right)^2$ with $h'_0$ defined as at the beginning of Section~\ref{sec:preliminaries}. It only remains to show that the problem $\mathscr{P}(q, f + \alpha, F)$ has the eigenvalues $\mu_n$. But first we show that the polynomials $f_\downarrow(\lambda)$ and $p(\lambda)$ coincide up to a constant factor. Arguing as in the proof of Lemma~\ref{lem:zero_sum} we have
\begin{equation*}
  \sum_{n=0}^\infty \frac{\lambda_n^k p(\lambda_n)}{\gamma_n} = \sum_{n=0}^\infty \frac{\lambda_n^k \Psi(\lambda_n)}{\pi \nu p(\lambda_n) \Phi'(\lambda_n)} = \frac{1}{2 \pi^2 \nu \iu} \lim_{N \to \infty} \int_{C_N} \frac{\lambda^k \left( \Psi(\lambda) - \Phi(\lambda) \right)}{p(\lambda) \Phi(\lambda)} \,\du \lambda = 0,
\end{equation*}
where $C_N$ is the same as in that proof. Now arguing as after Lemma~\ref{lem:zero_sum} we obtain that the non-leading coefficients of the polynomial $(-1)^d p(\lambda)$ satisfy the system~(\ref{eq:p_s}). Therefore $f_\downarrow(\lambda) = h'_0 p(\lambda)$.

Denote the eigenvalues of the boundary value problem $\mathscr{P}(q, f + \alpha, F)$ by $\widehat{\mu}_n$. They coincide with the zeros of the function
\begin{equation*}
  \widehat{\Psi}(\lambda) := F_\uparrow(\lambda) \psi(\pi, \lambda) - F_\downarrow(\lambda) \psi'(\pi, \lambda),
\end{equation*}
where $\psi(x, \lambda)$ is defined as in~(\ref{eq:phi_psi}). Using the results of Section~\ref{sec:two_problems}, we obtain
\begin{equation*}
  \widehat{\Psi}(\lambda_n) = \frac{\alpha f_\downarrow^2(\lambda_n) \Phi'(\lambda_n)}{\gamma_n} = \frac{\pi \nu p^2(\lambda_n) \Phi'(\lambda_n)}{\gamma_n} = \Psi(\lambda_n), \qquad n \ge 0.
\end{equation*}
This and the proofs of Lemmas~\ref{lem:zero_sum}, \ref{lem:infinite_product} and \ref{lem:G} show that $\left( \widehat{\Psi}(\lambda) - \Psi(\lambda) \right) / \Phi(\lambda)$ is an entire function satisfying the estimate
\begin{equation*}
  \frac{\widehat{\Psi}(\lambda) - \Psi(\lambda)}{\Phi(\lambda)} = O \left( \frac{1}{\sqrt{\lambda}} \right)
\end{equation*}
on $\bigcup_N C_N$ and hence by the maximum principle on the whole plane. Then the Liouville theorem implies that this function is identically zero. Thus $\widehat{\Psi}(\lambda) \equiv \Psi(\lambda)$ and hence $\widehat{\mu}_n = \mu_n$, $n \ge 0$.

So far, we have constructed two problems $\mathscr{P}(q, f, F)$ and $\mathscr{P}(q, f + \alpha, F)$ with the eigenvalues $\{ \lambda_n \}_{n \ge 0}$ and $\{ \mu_n \}_{n \ge 0}$ respectively, and such that the poles of $f$ are $i_1$-th, $i_2$-th, $\ldots$, $i_d$-th zeros of $\Phi(\lambda) - \Psi(\lambda)$. In order to prove that such a pair of problems is unique, we assume that the problems $\mathscr{P}(q, f, F)$ and $\mathscr{P}(\widetilde{q}, \widetilde{f}, \widetilde{F})$ have the eigenvalues $\{ \lambda_n \}_{n \ge 0}$, and the problems $\mathscr{P}(q, f + \alpha, F)$ and $\mathscr{P}(\widetilde{q}, \widetilde{f} + \widetilde{\alpha}, \widetilde{F})$ have the eigenvalues $\{ \mu_n \}_{n \ge 0}$. We also assume that the poles of $f$ (respectively, $\widetilde{f}$) are $i_1$-th, $i_2$-th, $\ldots$, $i_d$-th zeros of $\Phi - \Psi$ (respectively, $\widetilde{\Phi} - \widetilde{\Psi}$). Then $\Phi \equiv \widetilde{\Phi}$ and $\Psi \equiv \widetilde{\Psi}$ by Lemma~\ref{lem:infinite_product}, and hence $p \equiv \widetilde{p}$. Thus
\begin{equation*}
  \gamma_n = \frac{\pi \nu p^2(\lambda_n) \Phi'(\lambda_n)}{\Psi(\lambda_n)} = \frac{\pi \nu \widetilde{p}^2(\lambda_n) \widetilde{\Phi}'(\lambda_n)}{\widetilde{\Psi}(\lambda_n)} = \widetilde{\gamma}_n.
\end{equation*}
Therefore the uniqueness part of \cite[Theorem 4.7]{G17} implies that $q = \widetilde{q}$ a.e. on $[0, \pi]$, $f = \widetilde{f}$ and $F = \widetilde{F}$. Finally, $f = \widetilde{f}$ yields
\begin{equation*}
  \alpha = \pi \nu / \left( h'_0 \right)^2 = \pi \nu / \left( \widetilde{h}'_0 \right)^2 = \widetilde{\alpha}.
\end{equation*}

\appendix
\section{Auxiliary results}

In this appendix we prove two auxiliary lemmas used in the main body of the paper.

\begin{lemma} \label{lem:infinite_product}
  The characteristic function $\Phi(\lambda)$ of $\mathscr{P}(q, f, F)$ with $\ind f \ge 0$ has the infinite product representation
\begin{equation*}
  \Phi(\lambda) = -\prod_{n < L} (\lambda_n - \lambda) \prod_{n = L} \pi (\lambda_n - \lambda) \prod_{n > L} \frac{\lambda_n - \lambda}{(n - L)^2}
\end{equation*}
with
\begin{equation*}
  L := \frac{\ind f + \ind F}{2}.
\end{equation*}
\end{lemma}
\begin{proof}
The first-order asymptotics
\begin{equation} \label{eq:Phi_rough}
  \Phi(\lambda) = \lambda^{L + 1/2} \sin \left( \sqrt{\lambda} \pi + L \pi \right) + O\left( \lambda^L e^{|\im \sqrt{\lambda}\pi|} \right)
\end{equation}
was obtained in~\cite{G17} (see the proof of Lemma 2.2 therein). From Hadamard's theorem we obtain
\begin{equation*}
  \Phi(\lambda) = C \prod_{n = 0}^\infty \left( 1 - \frac{\lambda}{\lambda_n} \right) = C \prod_{n < L} \left( 1 - \frac{\lambda}{\lambda_n} \right) \prod_{n = L} \left( 1 - \frac{\lambda}{\lambda_n} \right) \prod_{n > L} \left( 1 - \frac{\lambda}{\lambda_n} \right).
\end{equation*}
Note also that, according to our assumption, $L$ is an integer or half-integer with $L \ge -1/2$ and this guarantees that $\lambda^{L + 1/2} \sin \left( \sqrt{\lambda} \pi + L \pi \right)$ is an entire function. We can combine infinite product representations for the sine and cosine functions into
\begin{equation*}
  \sin \left( \sqrt{\lambda} \pi + L \pi \right) = (-1)^{\lfloor L \rfloor} \prod_{n = L} \pi \sqrt{\lambda} \prod_{n > L} \left( 1 - \frac{\lambda}{\left( n - L \right)^2} \right).
\end{equation*}
The use of the identities
\begin{equation*}
  (-1)^{\lfloor L \rfloor} = - \prod_{n < L} (-1) \prod_{n = L} (-1), \qquad \lambda^{L + 1/2} = \prod_{n < L} \lambda \prod_{n = L} \sqrt{\lambda}
\end{equation*}
yields
\begin{equation*}
  \lambda^{L + 1/2} \sin \left( \sqrt{\lambda} \pi + L \pi \right) = - \prod_{n < L} (-\lambda) \prod_{n = L} (-\pi \lambda) \prod_{n > L} \left( 1 - \frac{\lambda}{\left( n - L \right)^2} \right).
\end{equation*}
Thus
\begin{equation*}
\begin{split}
  \frac{\Phi(\lambda)}{\lambda^{L + 1/2} \sin \left( \sqrt{\lambda} \pi + L \pi \right)} = - C &\prod_{n < L} \left( \frac{1}{\lambda_n} - \frac{1}{\lambda} \right) \prod_{n = L} \left( \frac{1}{\pi \lambda_n} - \frac{1}{\pi \lambda} \right) \\
  &\times \prod_{n > L} \frac{\left( n - L \right)^2}{\lambda_n} \prod_{n > L} \left( 1 + \frac{\lambda_n - \left( n - L \right)^2}{\left( n - L \right)^2 - \lambda} \right).
\end{split}
\end{equation*}
Taking the limit as $\lambda \to -\infty$ and using~(\ref{eq:Phi_rough}) and~(\ref{eq:lambda}) we obtain
\begin{equation*}
  C = - \prod_{n < L} \lambda_n \prod_{n = L} \pi \lambda_n \prod_{n > L} \frac{\lambda_n}{\left( n - L \right)^2},
\end{equation*}
which proves the lemma.
\end{proof}

To prove our next result we need a lemma of Marchenko and Ostrovskii.

\begin{lemma}[{\cite[Lemma~3.3]{MO75}, \cite[Lemma~3.4.2]{M77}}] \label{lem:Marchenko}
  For functions $u(z)$ and $v(z)$ to admit representations of the form
\begin{equation*}
  u(z) = \sin \pi z + A\pi\frac{4z}{4z^2-1}\cos \pi z + \frac{g_1(z)}{z}, \qquad v(z) = \cos \pi z - B\pi\frac{\sin \pi z}{z} + \frac{g_2(z)}{z},
\end{equation*}
where $g_1(z) = \int_0^{\pi} \widetilde{g}_1(t) \cos zt \,\du t$ and $g_2(z) = \int_0^{\pi} \widetilde{g}_2(t) \sin zt \,\du t$ with $\widetilde{g}_1$, $\widetilde{g}_2 \in \mathscr{L}_2[0,\pi]$ and $\int_0^{\pi} \widetilde{g}_1(t)\,\du t = 0$,
it is necessary and sufficient to have the form
\begin{equation*}
  u(z) = \pi z \prod_{n=1}^{\infty} n^{-2}(u_n^2 - z^2), \qquad u_n = n - \frac{A}{n} + \ell_2 \left( \frac{1}{n} \right),
\end{equation*}
\begin{equation*}
  v(z) = \prod_{n=1}^{\infty} \left( n - \frac{1}{2} \right)^{-2}(v_n^2 - z^2), \qquad v_n = n - \frac{1}{2} - \frac{B}{n} + \ell_2 \left( \frac{1}{n} \right).
\end{equation*}
\end{lemma}

Now we can prove

\begin{lemma} \label{lem:G}
  Let $\{ \eta_n \}_{n \ge 0}$ and $\{ \zeta_n \}_{n \ge 0}$ be sequences of real numbers having the asymptotics
\begin{equation*}
  \sqrt{\eta_n} = n - L + \frac{\sigma}{\pi n} + \ell_2 \left( \frac{1}{n} \right), \qquad \sqrt{\zeta_n} = n - L + O \left( \frac{1}{n} \right),
\end{equation*}
with an integer or half-integer $L \ge -1/2$ and a real $\sigma$, and let
\begin{equation*}
  G(\lambda) := -\prod_{n < L} (\eta_n - \lambda) \prod_{n = L} \pi (\eta_n - \lambda) \prod_{n > L} \frac{\eta_n - \lambda}{(n - L)^2}.
\end{equation*}
Then
\begin{equation*}
  G'(\zeta_n) = (-1)^n \left( n - L \right)^{2L} \left( \frac{\pi}{2} + \ell_2 \left( \frac{1}{n} \right) \right).
\end{equation*}
\end{lemma}
\begin{proof}
Using the asymptotics of $\eta_n$ and Lemma~\ref{lem:Marchenko} we obtain the representation of $G$ as the sum of four entire functions
\begin{equation*}
\begin{split}
  G(\lambda) = \lambda^{L + 1/2} \sin \left( \sqrt{\lambda} \pi + L \pi \right) &- \sigma \lambda^L \cos \left( \sqrt{\lambda} \pi + L \pi \right) \\
  &+ \lambda^L \int_0^{\pi} g(t) \cos \left( \sqrt{\lambda} t + L \pi \right) \,\du t + G_1(\lambda)
\end{split}
\end{equation*}
with $g \in \mathscr{L}_2[0, \pi]$, where $G_1(\lambda)$ is of the form
\begin{equation*}
  \int_0^{\pi} g(t) \left( p_1(\lambda) \frac{\sin \sqrt{\lambda} t}{\sqrt{\lambda}} + \frac{p_2(\lambda)}{\lambda} \cos \sqrt{\lambda} t \right) \,\du t + p_3(\lambda) \frac{\sin \sqrt{\lambda} \pi}{\sqrt{\lambda}} + p_4(\lambda) \frac{\cos \sqrt{\lambda} \pi}{4 \lambda - 1}
\end{equation*}
with some polynomials $p_1$, $p_2$, and $p_3$ of degree not exceeding $L$ and a polynomial $p_4$ of degree not exceeding $L + 1/2$. Therefore $G'_1(\lambda) = O\left( \lambda^{L - 1} e^{|\im \sqrt{\lambda}\pi|} \right)$ and consequently
\begin{multline*}
  G'(\lambda) = \frac{\pi}{2} \lambda^L \cos \left( \sqrt{\lambda} \pi + L \pi \right) + \left( L + \frac{\sigma \pi + 1}{2} \right) \lambda^{L - 1/2} \sin \left( \sqrt{\lambda} \pi + L \pi \right) \\
  -\frac{1}{2} \lambda^{L - 1/2} \int_0^{\pi} t g(t) \sin \left( \sqrt{\lambda} t + L \pi \right) \,\du t + O\left( \lambda^{L - 1} e^{|\im \sqrt{\lambda}\pi|} \right).
\end{multline*}
The statement of the lemma now follows from the asymptotics of $\zeta_n$.
\end{proof}

\section*{Acknowledgements}

The author thanks the anonymous referee for a very careful reading of the manuscript and for many useful suggestions.

\end{document}